\documentclass[11pt,reqno]{amsart}
  \usepackage{latexsym} 
  \usepackage[all]{xy}
  \usepackage{amsfonts} 
  \usepackage{amsthm} 
  \usepackage{amsmath} 
  \usepackage{amssymb}
  \usepackage{pifont}  
  \usepackage{enumerate}
  \usepackage{dcolumn}
  \usepackage{comment}
  \usepackage{hyperref}  
\newcolumntype{2}{D{.}{}{2.0}}
  \xyoption{2cell}

 \usepackage{tikz}
\usetikzlibrary{arrows}

   \def\<{{\langle}} 
  \def\>{{\rangle}}

  \def\note#1{{}}

  \def\note#1{}

  \def\beq{\begin{equation}} 
  \def\eeq{\end{equation}}

  \def\id{\mathrm{id}}

  \newcounter{zlist}

  \newcounter{blist}

  \newcounter{rlist}

\def\stac#1{\raise-.2cm\hbox{$\stackrel{\displaystyle\otimes}{\scriptscriptstyle{#1}}$}}

\def\cten#1{\raise-.2cm\hbox{$\stackrel{\displaystyle\reallywidehat{\otimes}}
{\scriptscriptstyle{#1}}$}}

  \def\Label#1{\label{#1}\ifmmode\llap{[#1] }\else 
  \marginpar{\smash{\hbox{\tiny [#1]}}}\fi} 
  \def\Label{\label}

  \newtheorem{proposition}{Proposition}[section]
  \newtheorem{lemma}[proposition]{Lemma} 
  \newtheorem{corollary}[proposition]{Corollary} 
  \newtheorem{theorem}[proposition]{Theorem} 
  
\theoremstyle{definition} 
  \newtheorem{definition}[proposition]{Definition}
    
  \newtheorem{example}[proposition]{Example}

  \theoremstyle{remark} 
  \newtheorem{remark}[proposition]{Remark}

  \newcounter{c} 
  \renewcommand{\[}{\setcounter{c}{1}$$} 
  \newcommand{\etyk}[1]{\vspace{-7.4mm}$$\begin{equation}\Label{#1} 
  \addtocounter{c}{1}} 
  \renewcommand{\]}{\ifnum \value{c}=1 $$\else \end{equation}\fi} 
   \numberwithin{equation}{section}

\newcommand{\Cc}{\mathcal{C}}

\newcommand{\sz}[2]{\sigma^{z}_{#1}(#2)}
\newcommand{\sn}[2]{\sigma^{1}_{#1}(#2)}
\newcommand{\tz}[2]{\tau^{z}_{#1}(#2)}
\newcommand{\tn}[2]{\tau^{1}_{#1}(#2)}

\def\*C{{}^*\hspace*{-1pt}{\Cc}}

\def\text#1{{\rm {\rm #1}}}

 \def\h{\mathbf{h}}

 \def\1{\mathbf{1}}

\def\id{\mathrm{id}}

\def\1\mathbf{1}

\def\|#1{\overline{#1}}

\def\h#1 {\hat{#1}}
\usepackage{scalerel,stackengine}
\stackMath
\newcommand\reallywidehat[1]{%
\savestack{\tmpbox}{\stretchto{%
  \scaleto{%
    \scalerel*[\widthof{\ensuremath{#1}}]{\kern.1pt\mathchar"0362\kern.1pt}%
    {\rule{0ex}{\textheight}}
  }{\textheight}%
}{2.4ex}}%
\stackon[-6.9pt]{#1}{\tmpbox}%
}
\begin{document}

\title[Novel non-involutive solutions of the YBE from (skew) braces]{Novel non-involutive solutions of the Yang-Baxter equation from (skew) braces}

\begin{abstract}
We produce novel non-involutive solutions of the Yang-Baxter equation coming from (skew) braces. These solutions are generalisations of the known ones coming from braces and skew braces,  and surprisingly in the case of braces they are not necessarily involutive.  In the case of two-sided (skew) braces one can assign such solutions to every element of the set. Novel bijective maps associated to the inverse solutions are also introduced. Moreover,  we show that the recently derived  Drinfeld twists of the involutive case are still admissible in the non-involutive frame and we identify the twisted $r$-matrices and twisted coproducts. We observe,  as in the involutive case  that the underlying quantum algebra is not a quasi-triangular bialgebra,  as one would expect, but a quasi-triangular quasi-bialgebra. The same applies to the quantum algebra of the twisted $r$-matrices, contrary to the involutive case.
\end{abstract}

\author{Anastasia Doikou}

\author{Bernard  Rybo{\l}owicz}

\address{(A.Doikou $\&$ B. Rybo{\l}owicz)
Department of Mathematics, 
Heriot-Watt University, 
Edinburgh EH14 4AS,  and Maxwell Institute for Mathematical Sciences, Edinburgh, UK}

\email{A.Doikou@hw.ac.uk,\ B.Rybolowicz@hw.ac.uk}


\baselineskip=15pt
\date\today
\maketitle

\section*{Introduction}
\noindent 
The idea of set-theoretic solutions of the Yang-Baxter equation (YBE) \cite{Baxter, Yang}
was first  introduced and studied by Drinfeld in early 90s
\cite{Drin} and ever since a significant progress has been made on this topic 
(see for instance \cite{ESS, Eti}). 
In 2007 Rump \cite{[26]}  identified certain algebraic structures called left braces and showed that with every left brace one 
can associate an involutive solution of the set-theoretic Yang-Baxter equation, and conversely from every involutive 
solution one can construct a left brace, such that the solution given by the brace, restricted to an appropriate subset,  
is a set-theoretic solution. This generated an 
increased research activity on left braces and  set-theoretic of the YBE (see for instance \cite{Ba, bcjo, CJV, [6],fc}), 
and in 2017 Guarnieri  and  Vendramin  \cite{GV}  extended Rump's construction to left skew braces in order to 
produce non-degenerate,  non-involutive solutions.
This generalization led to a trend of relaxing more conditions of braces to produce yet more general classes of solutions 
(see e.g. \cite{Catino, CMS, JKA, kava, KSV, Lebed, SVB}, \cite{GI3}--\cite{gateva}).

In the first part of the present study (Sections 1 and 2) we introduce a new way of constructing solutions of the 
YBE from left skew braces. In 
contrast to already known methods, this one allows one to associate to left (skew) braces
more than one solution,  not necessarily involutive.  To every two-sided (skew) brace we can associate as many solutions 
as there are elements. These solutions are not necessarily all different, but it is quite common for them to be distinct, 
(see Example \ref{ex:fractions} and \ref{ex:cyclicbraces}).  Due to the dependence of the solution on the choice 
of an element of the set, we obtain
parameter dependent families of solutions. The striking observation here is that even in the case of braces 
one may obtain non-involutive solutions!  In fact, we show that for generic values of the parameter there is no map to relate our parametric solutions with the known solutions appearing in \cite{[26], GV}.

In the second part (Sections 3 and 4) we study the underlying quantum algebras 
associated to non-involutive set-theoretic solutions 
and discuss the notion of  admissible Drinfeld twist.
Admissible twists ${\mathcal F}$, which link distinct  (quasi-triangular) Hopf or quasi-Hopf algebras 
were  introduced by Drinfeld in \cite{Drintw}. 
Whenever the notion of the  twist is discussed in Drinfeld's original work and in the literature in general
a restrictive action of the co-unit on the twist is almost always assumed, i.e.
 $(\mbox{id} \otimes \epsilon) {\mathcal F} = (\epsilon \otimes \mbox{id} ) {\mathcal F}=1_{\mathcal A}$ (${\mathcal A}$ 
is the associated quantum algebra).  We should note however that Drinfeld in \cite{Drintw}  uses certain simple twists 
without this restricted counit action to twist quasi-bialgebras with nontrivial unit constraints to 
quasi-bialgebras with trivial unit constraints.  

Recently it was shown in \cite{Doikoutw} that all involutive, set-theoretic solutions of the YBE 
can be obtained from the permutation operator via a suitable twist that was explicitly derived 
and its admissibility was proven.
It was noted in \cite{Doikoutw}, using a special class of set-theoretic solutions,
that the corresponding quantum algebra was not co-associative and the related associator 
was derived. This was indeed the first hint that set-theoretic solutions of the YBE give rise to quasi-bialgebras.
Later the idea of non co-associativity for involutive solutions was further explored in \cite{DOGHVL} 
and the generic action of the co-unit on the twist was considered,  leading  
to the conclusion that the underlying quantum algebra for involutive 
set-theoretic solutions (see also \cite{DoiSmo1, DoiSmo2}) is a quasi-triangular quasi-bialgebra.  
We extend these findings here to the non-involutive case.

We describe below in more detail the outline and the main findings of this study.

\begin{enumerate}

\item In Section 1 we recall basic definitions and known results about (skew) 
braces and set-theoretic solutions of the Yang-Baxter equation.

\item In Section 2 we introduce novel set-theoretic solutions of the YBE
associated to generalized left skew braces $B$, that depend on an extra parameter $z\in B,$ 
(Theorem \ref{prop:main}) i.e.  we obtain a parameter dependent family of solutions.
For $z=1$, one recovers the solution given by Guarnieri and Vendramin in \cite{GV}. 
We also show that our $z$-deformed  solution is involutive if and only if the skew brace 
is an abelian additive group and $z$ belongs to the socle of the brace,
that is the solution reduces to Rump's solution given in \cite{[26]}.  One of the most striking findings in this section is 
Corollary 3.5, which states that if $B$ is a skew brace with certain properties, then for any element $z\in B$, 
we can associate a $z$-deformed solution. 
Another key observation, as already mentioned,  is that even in the case of braces 
one may obtain non-involutive solutions.  
We also introduce novel bijective maps associated to the inverse $r$-matrices (also solutions of the YBE).
We conclude the section with 
several examples which show that not all deformed solutions are different, but it is not common 
for two deformed solutions 
to be the same.  
We are able to present examples of non-involutive solutions associated to braces.

\item In Section 3 we briefly review key definitions about the notions of the quasi-triangular 
quasi-bialgebras and Drinfeld twists as well as some of the main findings of \cite{DOGHVL}, 
useful for the analysis of Section 4.

\item In Section 4 we discuss the  quantum algebras and the notion of Drinfeld twists for non-degenerate, 
non-involutive solutions of the YBE.  Specifically,  we first provide a brief review
on the tensor notation for set-theoretic $r$-matrices (see also e.g. \cite{DoiSmo1}) 
as well as the main results on admissible Drinfeld twists and the associated quasi-bialgebras
for involutive set-theoretic solutions \cite{Doikoutw, DOGHVL}.  We then move one to our aim,  
which is the generalization of the results of \cite{Doikoutw, DOGHVL}  to the non-involutive scenario and the characterization of the associated quantum algebra as  a quasi-bialgebra.
This is achieved by introducing  certain  families of elements, which after suitable twisting become group-like elements. 
The form and the coproduct structure of these elements are inspired by tensor representations of the so called RTT algebra \cite{FadTakRes, Doikoutw}. The existence of these families of elements is fundamental in allowing us to conclude 
that for any set-theoretic solution of the YBE, the underlying 
quantum algebra is a quasi-bialgebra. 

\end{enumerate}

\section{Preliminaries}
\noindent 
We begin the section by recalling basic definitions and ideas about  
set-theoretic solutions of the Yang-Baxter equation and (skew) braces.
For a set-theoretic solution of the braid equation, we will use the notation $(X, \check{r})$, 
instead of the usual notation  $(X,r)$,  to be consistent with the notation used in quantum integrable systems.

Let $X=\{x_{1}, \ldots, x_n\}$ be a set and ${\check r}:X\times X\rightarrow X\times X$.
Denote 
\begin{equation}
{\check r}(x,y)=(\sigma _{x}(y), \tau _{y}(x)). \label{setY}
\end{equation}
We say that $\check r$ is non-degenerate if 
$\sigma _{x}$ and $\tau _{y}$ 
are bijective maps, and
$(X, {\check r})$ is a set-theoretic solution of the braid equation  if
\begin{equation}
({\check r}\times \mbox{id})( \mbox{id} \times {\check r})({\check r}\times \mbox{id})=(\mbox{id}\times 
{\check r})({\check r}\times \mbox{id})(\mbox{id}\times {\check r}).\label{braid}
\end{equation}
The map $\check r$ is called involutive if $\check r^2=\mbox{id}.$

The following notion of homomorphism and isomorphism between solutions will be useful.
\begin{definition}\cite[Section 3]{[6]}\label{def11}
Let $X,S$ be sets and $\check r:S\times S\to S\times S$ and $\hat r:X\times X\to X\times X$ be two solutions of the braid equation. Then a function $f: S\to X$ is called a {\em homomorphism of Yang-Baxter solutions} if it satisfies the following equality
$$
(f\times f)\check r=\hat r(f\times f).
$$
We say that $f$ is an {\em isomorphism} if $f$ is a bijection.
\end{definition}

We also introduce the map $r: X\times X\rightarrow X\times X,$ such that $r = \check r\pi,$ 
where $\pi: X\times X\rightarrow X\times X$ 
is the flip map: $\pi(x,y) = (y,x).$ Hence, $r(y,x) =( \sigma_x(y), \tau_y(x)),$ and it satisfies the YBE:
\begin{equation}
r_{12}\ r_{13}\  r_{23} =r_{23}\ r_{13}\ r_{12},  \label{YBE}
\end{equation} 
where we denote $r_{12}(y,x,z) = (\sigma_x(y), \tau_y(x),z),$ $r_{23}(z,y,x) = (z, \sigma_x(y), \tau_y(x))$ 
and  $r_{13}(y,z,x) = (\sigma_x(y), z,\tau_y(x)).$

We note that a function satisfying \eqref{braid}, is usually called in literature a set-theoretic solution of the Yang-Baxter equation or the pair $(X,\check{r})$ is called a braided set \cite{ESS,Gateva,gateva}. Also, a function satisfying \eqref{YBE}, can be found in literature under the name set-theoretic solution of the quantum Yang-Baxter equation (QYBE) \cite{ESS, [25],[26]}. In this paper a pair $(X, \check r),$ which satisfies (\ref{braid} is called a set-theoretic solution of the braid equation, whereas a pair $(X, r),$ which satisfies (\ref{YBE}) is called a set-theoretic solution of the YBE.

We recall the definitions of the algebraic structures that provide set-theoretic solutions of the braid equation, 
such as left skew braces and  braces.  We also present some key properties associated 
to these structures that will be useful when formulating some of the main findings of the present study, summarized in Section 4.

\begin{definition}[\cite{[6], GV}]
A {\it left skew brace} is a set $B$ together with two group operations $+,\circ :B\times B\to B$, 
the first is called addition and the second is called multiplication, such that for all $ a,b,c\in B$,
\begin{equation}\label{def:dis}
a\circ (b+c)=a\circ b-a+a\circ c.
\end{equation}
If $+$ is an abelian group operation $B$ is called a {\it left brace}.
Moreover, if $B$ is a left skew brace and for all $ a,b,c\in B$ $(b+c)\circ a=b\circ a-a+c \circ a$, then $B$ is called a 
{\it skew brace} or more commonly in literature {\it two-sided skew brace}. Analogously if $+$ is abelian and $B$ is a skew brace, then $B$ is called a {\it brace}.
\end{definition}

 The additive identity of a left skew brace $B$ will be denoted by $0$ and the multiplicative identity by $1$. 
In every left skew brace $0=1$.

Rump showed the following powerful theorem for  involutive set-theoretic solutions.
\begin{theorem}\label{Rump}(Rump's theorem, \cite{[25], [26], [6]}).  
Assume  $(B, +, \circ)$ is a left brace. If the map  $\check r_B: B\times B \to B \times B$ is defined as 
${\check r}_B(x,y)=(\sigma _{x}(y), \tau _{y}(x))$, where $\sigma _{x}(y)=x\circ y-x$, $\tau _{y}(x)=t\circ x-t$, and $t $ 
is the inverse of $\sigma _{x}(y)$ in the circle group $(B, \circ ),$ then  $(B, \check r_B)$ is an involutive,  
non-degenerate solution of the braid equation.\\
Conversely,  if $(X,\check r)$ is an involutive, non-degenerate solution of the braid equation, then there exists a left brace $(B,+, \circ)$ 
(called an  underlying brace of the solution $(X, \check r)$) such that $B$ contains $X,$ $\check r_B(X\times X )\subseteq X \times X,$
and the map $\check r$ is isomorphic to the restriction of $\check r_B$ to $X \times X.$ Both the additive $(B, +)$ 
and multiplicative $(B,\circ)$ groups of the left brace $(B,+, \circ)$ are generated by $X.$
\end{theorem}
The following fact was also noticed by Rump.
\begin{remark} \label{nilpotent}
 Let $(N, +, \cdot)$ be an associative ring.  If for $a,b\in N$ we define 
 \[a\circ b=a\cdot b+a+b,\] then $(N, +, \circ )$ is a brace if and only if $(N, +, \cdot)$ is a radical ring.
\end{remark}

Guarnieri and Vendramin \cite{GV}, extended Rump's result to left skew braces and non-degenerate, non-involutive solutions.
\begin{theorem}[\cite{GV} Theorem 3.1]\label{thm:GV}
Let $B$ be a left skew brace, then the map $\check{r}_{GV}:B\times B\to B\times B$ given  for all $a,b\in B$ by
$$
\check{r}_{GV}(a,b)=(-a+a\circ b,\ (-a+a\circ b)^{-1}\circ a\circ b)
$$
is a non-degenerate solution of the braid equation.  
\end{theorem}

\section{Extended non-involutive solutions of the YBE from (skew) braces}

 \noindent We are going to consider in this section some generalized version of the set-theoretic solution (\ref{setY})
by introducing some kind of ``$z$-deformation''.  Indeed, let $z\in X$ be fixed, then we denote 
\begin{equation}
\check r_z (x,y) = (\sigma^z_{x}(y),\tau^z_{y}(x)). \label{genr}
\end{equation}
We say that $\check r$ is non-degenerate if $\sigma_x^z$ and $\tau_y^z$ are bijective maps.
We review below the constraints arising by requiring $(X, \check r_z)$ to be a solution 
of the braid equation (\cite{Drin,  ESS, [25], [26]}). Let,
\[({\check r}\times \mbox{id})(\mbox{id}\times {\check r})({\check r}\times \mbox{id})(\eta, x, y)=(L_1, L_2, L_3),\]
\[(\mbox{id}\times {\check r})({\check r}\times \mbox{id})(\mbox{id}\times {\check r})(\eta,x,y)= (R_1, R_2, R_3),\]
where, after employing expression  (\ref{genr}) we identify:
\begin{eqnarray}
L_1 = \sigma^z_{\sigma^z_{\eta}(x)}(\sigma^z_{\tau^z_{x}(\eta)}(y)),\quad L_2 = 
\tau^z_{\sigma^z_{\tau^z_x(\eta)}(y)}(\sigma^z_{\eta}(x)), \quad L_3 =\tau^z_y(\tau^z_x(\eta)),\nonumber
\end{eqnarray}
\begin{eqnarray}
R_1=\sigma^z_{\eta}(\sigma^z_x(y)), \quad R_1=\sigma^z_{\tau^z_{\sigma^z_x(y)}(\eta)}(\tau^z_{y}(x)), \quad 
R_3= \tau^z_{\tau^z_{y}(x)}(\tau^z_{\sigma^z_{x}(y)}(\eta)). \nonumber
\end{eqnarray}
And by requiring $L_i =R_i,$ $i\in \{1,2,3\}$ we obtain the following fundamental constraints for the associated maps:
\begin{eqnarray}
&&  \sigma^z_{\eta}(\sigma^z_x(y))= \sigma^z_{\sigma^z_{\eta}(x)}(\sigma^z_{\tau^z_{x}(\eta)}(y)), \label{C1}\\
&& \tau^z_y(\tau^z_x(\eta)) = \tau^z_{\tau^z_{y}(x)}(\tau^z_{\sigma^z_{x}(y)}(\eta)), \label{C2}\\
&&  \tau^z_{\sigma^z_{\tau^z_x(\eta)}(y)}(\sigma^z_{\eta}(x))= \sigma^z_{\tau^z_{\sigma^z_x(y)}(\eta)}(\tau^z_{y}(x)). \label{C3}
\end{eqnarray}
Note that the constraints above are essentially the ones for the set-theoretic solution (\ref{setY}), given that $z$ 
is a fixed element of the set,  i.e.  for different elements $z$ we obtain in principle distinct solutions of the braid equation.

We will introduce in what follows suitable algebraic structures that satisfy 
the fundamental constraints above,  i.e.  provide solutions of the braid equation 
and generalize the findings of Rump and Guarnieri $\&$ Vendramin. 
Before we state and prove our main Theorem \ref{prop:main},  we first show some useful properties 
of left skew braces.
\begin{lemma}\label{lem:ter}
Let $B$ be a set with two group operations $+,\circ $ with the same neutral element $1$. 
Then the condition \eqref{def:dis} is equivalent to the following condition: 
$$
a\circ (b-c+d)=a\circ b-a\circ c+a\circ d,
$$
for all $a,b,c,d\in B$
\end{lemma}
\begin{proof} This follows from \cite{Brz:tru}. Indeed,  let us assume that \eqref{def:dis} holds,  and recall that in any skew brace $0=1, $ then we observe
\begin{eqnarray}
a\circ 1 =a \Rightarrow a \circ (x - x) = a \Rightarrow a\circ(-x) = a - a\circ x +a,\nonumber
\end{eqnarray}
and consequently $\forall a,b,c,d\in B$, we have
$$
\begin{aligned}
&a\circ (b-c+d)=a\circ (b-c)-a+a\circ d=a\circ b -a +a\circ (-c)- a+ a\circ d\\ &=
a\circ b -a+a -a\circ c+a- a+ a\circ d=a\circ b-a\circ c+a\circ d.
\end{aligned}
$$
Conversely,  we observe that 
$$
a\circ (b+c)=a\circ (b-1+c)=a\circ b-a+a\circ c,
$$
so \eqref{def:dis} holds.
\end{proof}
\begin{remark}
Similarly, one can rephrase the right distributivity in the brace, i.e. for all $ a,b,c,d\in B$,
$$
(b+c)\circ a=b\circ a-a+c\circ a \iff (b-c+d)\circ a=b\circ a-c\circ a+d\circ a.
$$
\end{remark}

We introduce below the notion of uniformly distributive of $u$-distributive elements.
\begin{definition}
    Let $B$ be a left skew brace. We say that $z\in B$ is {\em $u$-distributive}, if for all $a,b,c\in B,$
    \begin{equation}(a-b+c)\circ z=a\circ z-b\circ z+c\circ z.\label{regular}\end{equation}
\end{definition}

\begin{remark}
    Notice that if $B$ is a skew brace, then every element of $B$ is $u$-distributive.
\end{remark}

\begin{proposition}\label{lem:long}
Let $B$ be a left skew brace and $z\in B$ be $u$-distributive. For every $a\in B,$ let $\sigma^z_a:B\to B$ and $\tau^{z}_a:B\to B$ be the maps defined as
\begin{equation}\sigma^z_a(b):=a\circ b-a\circ z+z\quad and\quad \tau^{z}_b(a):=
\sz{a}{b}^{-1}\circ a\circ b,\  for\  all\ b\in B,\label{sigmatau}
\end{equation}
 where $\sz{a}{b}^{-1}$ is the inverse of $\sz{a}{b}$ in $(B,\circ)$. Then the following conditions are satisfied for all $ a,b,c,d\in B:$
\begin{enumerate}
\item The set $\{z\in B\ |\ \forall{a,b,c\in B}\ (a-b+c)\circ z=a\circ z-b\circ z+c\circ z\}$ is a subgroup of $(B,\circ)$, \label{lem:long:eq:4}
\item $\sz{a}{b-c+d}=\sz{a}{b}-\sz{a}{c}+\sz{a}{d}$, \label{lem:long:eq:1}
\item $\sz{a}{\sz{b}{c}}=\sz{a\circ b}{c} $, \label{lem:long:eq:2}
\item $a\circ \sz{b}{c}=\sz{a\circ b}{c}-z+a\circ z$,\label{lem:long:eq:3}
\item $\sz{a}{b}\circ\tz{b}{a}= a\circ b$\label{lem:long:eq:5}
\item $\sigma^z_{a}(b) \circ \sigma^z_{\tau^z_{b}(a)}(c)= \sigma^z_{a}(\sz{b}{c}) \circ 
\sigma^z_{\tz{\sz{b}{c}}{a}}(\tz{c}{b}).$\label{lem:long:eq:6}
\item The maps $\sigma^z_a$ and $\tau^z_{a}$ are bijective for all $a\in B.$\label{lem:long:eq:7}
\end{enumerate}
\end{proposition}

\begin{proof}
Let $a,b,c,d \in B$,  and we recall  Lemma 2.1:

\begin{eqnarray}
\sz{a}{b-c+d}&=&a\circ (b-c+d)-a\circ z+z=a\circ b-a\circ c+a\circ d-a\circ z+z\nonumber \\ 
&=&a\circ b-a \circ z+a \circ z-a \circ c+a \circ d-a \circ z+z \nonumber \\ 
&=&a \circ b-a \circ z+z-z+a\circ z-a \circ c+a \circ d-a \circ z+z\nonumber \\ 
&=&\circ b-a\circ z+z-(a\circ c-a\circ z+z)+a\circ d-a\circ z+z\nonumber \\ 
&=&a\circ b-a\circ z+z-(a\circ c-a\circ z+z)+a\circ d-a\circ z+z\nonumber \\ 
&=&\sz{a}{b}-\sz{a}{c}+\sz{a}{d},\nonumber 
\end{eqnarray}
\begin{eqnarray}
\sz{a}{\sz{b}{c}}&=&\sz{a}{b\circ c-b\circ z+z}=a\circ (b\circ c-b\circ z+z)-a\circ z+z\nonumber\\ 
&=&a\circ b\circ c-a\circ b\circ z+a\circ z-a\circ z+z \nonumber \\ 
&=&a\circ b\circ c-a\circ b\circ z+z=\sz{a\circ b}{c}, \nonumber 
\end{eqnarray}
\begin{eqnarray}
a\circ \sz{b}{c}&=&a\circ (b\circ c-b\circ z+z)=a\circ b\circ c-a\circ b\circ z+a\circ z \nonumber  \\ 
&=&a\circ b\circ c-a\circ b\circ z+z-z+a\circ z=\sz{a\circ b}{c}-z+a\circ z,\nonumber
\end{eqnarray}
\begin{eqnarray}
(a-b+c)\circ z^{-1}&=&(a\circ z^{-1}\circ z-b\circ z^{-1}\circ z+c\circ z^{-1}\circ z)\circ z^{-1} \nonumber \\ 
&=&(a\circ z^{-1}-b\circ z^{-1}+c\circ z^{-1})\circ z\circ z^{-1}\nonumber \\ 
&=&(a\circ z^{-1}-b\circ z^{-1}+ c\circ z^{-1}),\nonumber
\end{eqnarray}
\begin{eqnarray}
\sz{a}{b}\circ\tz{b}{a}&=&\sz{a}{b}\circ \sz{a}{b}^{-1}\circ a\circ b= a\circ b,\nonumber 
\end{eqnarray}
thus,  properties (1)-(5) hold.

To show (6)  we observe that using Lemma \ref{lem:ter} and the 
fact that multiplying by $z$ distributes from the right side,  we get for all $a,b,c\in B$
$$
\begin{aligned}
\sigma^z_{a}(b) \circ \sigma^z_{\tau^z_{b}(a)}(c)&=\sz{a}{b}\circ(\tz{b}{a}\circ c-
 \tz{b}{a}\circ z+z)\\ &=\sz{a}{b}\circ \tz{a}{b}\circ c- \sz{a}{b}\circ \tz{a}{b}\circ z+\sz{a}{b}\circ z\\ &=
a\circ b\circ c-a\circ b\circ z+ \sz{a}{b}\circ z\\ &=a\circ b\circ c-a\circ b\circ z+ 
(a\circ b-a\circ z+z)\circ z\\ &=a \circ b \circ c- a \circ z \circ z+z\circ z,
\end{aligned}
$$

By substituting $b$ with $\sigma^z_b(c)$ and $c$ with $\tau^z_c(b)$, by using Proposition \ref{lem:long}\eqref{lem:long:eq:5}, we immediately get 

$\sigma^z_{a}(\sz{b}{c}) \circ \sigma^z_{\tz{\sz{b}{c}}{a}}(\tz{c}{b})=a \circ\sigma^z_b(c) \circ \tau^z_c(b)- a \circ z \circ z+z\circ z=a \circ b \circ c- a \circ z \circ z+z\circ z.$

Thus, $\sigma^z_{a}(b) \circ \sigma^z_{\tau^z_{b}(a)}(c)=\sigma^z_{a}(\sz{b}{c}) 
\circ \sigma^z_{\tz{\sz{b}{c}}{a}}(\tz{c}{b})$, and (6) holds.

For (7), observe that both maps are injective as
\begin{eqnarray}
&&\sigma_x^z(y_1) = \sigma_x^z(y_2)  \Leftrightarrow x\circ y_1 - x\circ z +z = 
x \circ y_2- x \circ z +z \nonumber \Leftrightarrow y_1 =y_2,\nonumber\\
&&\tau^z_{y}(x_1) = \tau^z_{y}(x_2) \Leftrightarrow t_1\circ x_1 \circ z  = t_2\circ x_2 \circ z \Leftrightarrow  
t_1\circ x_1  = t_2\circ x_2, \label{cc1} \nonumber
\end{eqnarray}
where recall $t_i = \sigma^z_{x_i}(y)^{-1},$ $i \in \{1,2\}$. Thus, 
$t_1\circ x_1=t_2\circ x_2\iff x_1^{-1}\circ t_1^{-1}=x_2^{-1}\circ t_2^{-1}$, 
which leads to
$y-z+x_1^{-1}\circ z=y-z+x_2^{-1}\circ z,$ and hence $x_1=x_2$, so both 
maps are injective.  

To show that the maps are surjective, we observe that for all $ c\in B$, fixed $y,z\in B$, and $h=(y^{-1}-z^{-1}+c^{-1}\circ z^{-1})$,
$$
\sz{x}{x^{-1}\circ(c-z+x\circ z)}=c-z+x\circ z-x\circ z+z=c
$$
$$
\begin{aligned}
\tz{y}{h^{-1}\circ y^{-1}}&=(h^{-1}\circ y^{-1}\circ y-h^{-1}\circ y^{-1}\circ z+z)^{-1}\circ h^{-1}\circ y^{-1}\circ y\\ &=
\left(h^{-1}\circ(-y^{-1}\circ z+h\circ z)\right)^{-1}\circ h^{-1}\\
&=(-y^{-1}\circ z+h\circ z)^{-1}\circ h\circ h^{-1}=(-y^{-1}\circ z+h\circ z)^{-1}\\
&=\left (-y^{-1}\circ z+(y^{-1}-z^{-1}+c^{-1}\circ z^{-1})\circ z \right)^{-1}\\
&=\left (-y^{-1}\circ z+y^{-1}\circ z-1+c^{-1} \right)^{-1}=(c^{-1})^{-1}=c.
\end{aligned}
$$
Thus both $\tau^z_y$ and $\sigma^z_x$ are bijections.
\end{proof}

\noindent We should note that even though the definition of $\sigma^z_x$ might seem like a lucky guess, 
there is some intuition behind it. The map is closely related to trusses, see \cite{Brz:tru}. 
The map $\sigma^z$ appears in \cite{Brz:par}, denoted by $\lambda^z$, as a map that 
defines a paragon, that is a congruence class, see \cite{BrzRyb:con}. 
Thus, these solutions connect, in some particular way not yet clear to us, the structure of 
quotient of a brace with  the Yang Baxter-equation.

We may now proceed in proving the following main theorem.
\begin{theorem}\label{prop:main}
Let $B$ be a left skew brace, let $z\in B$ be $u$-distributive, and let $\sigma^z_a:B\to B,$ and $\tau^z_a:B\to B$ be the maps \eqref{sigmatau} defined in Proposition \ref{lem:long}. Suppose the map $\check{r}_z:B\times B\to B\times B$ is defined by
$$
\check{r}_z(a,b)=(\sz{a}{b},\tz{b}{a}),\quad a,b\in B.
$$
Then $\check{r}_z$ is a non-degenerate solution of the braid equation.

\end{theorem}
\begin{proof}To prove this we need to show that the maps $\sigma, \tau$ satisfy the constraints (\ref{C1})-(\ref{C3}).  
To achieve this we use Lemma \ref{lem:ter} and properties (2),  (5) and (6) from Proposition \ref{lem:long}. 

Indeed, from Proposition \ref{lem:long}, (2) and (5), it follows that (\ref{C1}) is satisfied, i.e. 
\[\sz{\eta}{\sz{x}{y}}=\sz{\sz{\eta}{x}}{\sz{\tz{x}{\eta}}{y}}.\]

We observe that
\begin{equation}
\tau^z_{b}(\tau^z_a(\eta)) = T\circ \tau^z_{a}(\eta)\circ b = T \circ t \circ \eta \circ a \circ b = 
T \circ t \circ \eta \circ \sz{a}{b} \circ \tz{b}{a}, \nonumber
\end{equation}
where $T= \sz{\tz{a}{\eta}}{b}^{-1}$ and $t = \sigma^z_{\eta}(a)^{-1} $ (the inverse in the  circle group). 
Due to (6),  (2),  (5)  of Proposition \ref{lem:long},  we then conclude that 
\[\tau^z_{b}(\tau^z_a(\eta)) =\tau^z_{\tz{b}{a}}(\tau^z_{\sz{a}{b}}(\eta)),\] 
so (\ref{C2}) is also satisfied. 

To prove (\ref{C3}),  we first employ (6) of  Proposition \ref{lem:long},  and then use the definition of $\tau$, 
$$
\sigma^z_{\tau^z_{\sigma^z_x(y)}(\eta)}(\tau^z_{y}(x))=\sz{\eta\circ x}{y}^{-1} \circ\sigma^z_{\eta}(x) \circ 
\sigma^z_{\tau^z_{x}(\eta)}(y)=\tz{\sz{\tz{x}{\eta}}{y}}{\sz{\eta}{x}}.
$$
Thus, (\ref{C3}) is satisfied, and $\check{r}_z(a,b)=(\sz{a}{b},\tz{b}{a})$ is a solution of braid equation.

The non-degeneracy follows by the Proposition \ref{lem:long} \eqref{lem:long:eq:7}.
\end{proof}

\begin{corollary}
Let $B$ be a skew brace. Then for all $ z \in B$, $\check r_z$ defined as in Theorem \ref{prop:main} 
is a set-theoretic solution of the braid equation.
\end{corollary}

\begin{remark} \label{rem2} Notice that for $z=1,$  Guarnieri-Vedramin skew braces are recovered; if in addition $(B, +)$ is an abelian group Rump's braces are recovered and  $\check r_{z=1}$ 
becomes involutive, due to: $\sigma^z_{\sigma^z_x(y)}(\tau^z_y(x)) = x$ and $\tau^z_{\tau^z_y(x)} (\sigma^z_x(y)) =y.$ 
In the general case $z \neq 1$ the solutions are not involutive anymore given that, although 
still $\sigma^z_x(y) \circ \tau_y^z(x) = x\circ y$ holds, 
$\sigma^z_{\sigma^z_x(y)}(\tau^z_y(x))\neq x,$ $\tau^z_{\tau^z_y(x)} (\sigma^z_x(y)) \neq y.$ 

The following Lemmata explain when our solutions are involutive.
\begin{lemma}\label{lem:trojkat}
Let $B$ be a left skew brace, let $z\in B$ be $u$-distributive. Then 
$$\sz{a}{b}=\sigma^{1}_{a}(b),\ for\  all\  a,b\in B$$  if and only if  $a\circ z=z+a$ for all $ a\in B$. In this case the solution $(B,\check{r}_z)$ coincides with $(B,\check{r}),$ the canonical solution on $B.$
\end{lemma}
\begin{proof}
Observe that for all $ a,b\in B$ such that $a\circ z=z+a$,
$$
\sz{a}{b}=a\circ b-a\circ z+z=a\circ b-(z+a)+z=a\circ b-a-z+z=a\circ b-a=\sn{a}{b}.
$$
Conversely, if $\sz{a}{b}=\sn{a}{b}$, then
$$
a\circ z=z-\sn{a}{b}+a\circ b=z-(a\circ b-a)+a\circ b=z+a-a\circ b+a\circ b=z+a.
\hfill \qedhere
$$
\end{proof}

\begin{lemma}\label{lem:inv}
Let $B$ be a left skew brace. The solution $\check r_{z}$ is involutive if and only if $B$ is 
a left brace and for all $ a,b\in B,$ $\sz{a}{b}=\sigma^{1}_{a}(b)$.
\end{lemma}
\begin{proof}
Let $B$ be a left brace and $\sz{a}{b}=\sn{a}{b}$ for all $a,b\in B$. Then
$$
\begin{aligned}
\sz{\sz{a}{b}}{\tz{b}{a}}&=\sn{\sn{a}{b}}{\tz{b}{a}}=a\circ b-\sn{a}{b}=a\circ b-(a\circ b-a)=a.
\end{aligned}
$$
Thus, we get that $\check r^2_{z}=(a,a^{-1}\circ a\circ b)=(a,b)$, i.e. $r_{z}$ is involutive.

Conversely, let us assume that $\check r_{z}$ is involutive, that is, for all $ a,b\in B$,
$$\sz{\sz{a}{b}}{\tz{b}{a}}=a.$$
If $a=1$, then $\sz{b}{1}=1$. Hence, $b-b\circ z+z=1$, and
$b\circ z=z+b$, so by Lemma \ref{lem:trojkat}, $\sz{a}{b}=\sn{a}{b}$. Then,
$$\sn{\sn{a}{b}}{\tn{b}{a}}=a\circ b-(a\circ b-a)=a,$$
and $a\circ b+a=a+a\circ b$, for all $ a,b\in B$. Now, let $b=a^{-1}\circ c$ for any $c\in B$, 
then $a+c=c+a$, $(B,+)$ 
is an abelian group, and thus $B$ is a left brace.
\end{proof}
\begin{remark}
In the case of a brace, Lemma \ref{lem:inv} and Lemma \ref{lem:trojkat}, say that $\check r_z$ is 
involutive if and only if $z\in \mathrm{Soc}(B):=\{b\in B\ |\ \forall{a\in B}\  a\circ b=a+b\}$. Moreover,  
for all $ z,z'\in \mathrm{Soc}(B)$ $\check r_z=\check r_{z'}$, that is there is only 
one involutive solution associated that way with a brace.
\end{remark}

Note that although $\check r_z$ is not involutive, $\check r_z \check r_z^T = \check r^T_z \check r_z =\mbox{id},$ where $^T$ denotes total transposition, i.e.  $\check r_z^T: (\sigma^z_x(y), \tau^z_{y}(x)) \mapsto (x, y)$; similarly $r_z r_z^T =r^T_z r_z=\mbox{id}.$ Recall  that $\check r_z$ satisfies the braid equation and $r_z$ satisfies the YBE.
\end{remark}

In the following Proposition we provide the explicit expressions of the inverse $\check r_z$-matrices as well as the corresponding  bijective maps.
\begin{proposition} \label{rem3} Let $\check r_z,\ \check r_{\hat z}^*: X \times X$ be solutions of the braid equations, such that
$\check r^*: (x,y )\mapsto (\hat \sigma^{z}_x(y),  \hat \tau^{z}_y(x)),$ 
$\check r: (x,y) \mapsto (\sigma^{z}_x(y),   \tau^{z}_y(x)).$

\begin{enumerate} 
\item $\check r^* = \check r^{-1}$ if and only if 
\begin{eqnarray}
\hat \sigma^{ z}_{\sigma_x^z(y)}(\tau_y^z(x)) =x, \  \hat \tau^{ z}_{\tau^z_{y}(x)}(\sigma^z_{x}(y)) = y \ \mbox{and} \ \sigma^{z}_{\hat \sigma_x^{z}(y)}(\hat \tau_y^{z}(x)) =x,  \ \tau^{z}_{\hat \tau^{ z}_{y}(x)}(\hat \sigma^{ z}_{x}(y)) = y.   \label{k2}
\end{eqnarray}

\item  If $\sigma^z_x(y) = x\circ y - x\circ z+z,$ $\ \tau^z_y(x) = \sigma^z_x(y)^{-1}\circ x \circ y,$ then $\hat\sigma^{z}_x(y)= -x\circ z^{-1} + x\circ y \circ z^{-1}, $ $\ \hat \tau^z_x(y) =\hat\sigma^{z}_x(y)^{-1}\circ x \circ y.$
\end{enumerate}

\end{proposition}
\begin{proof} The proof is straightforward.  We prove the two parts of Proposition \ref{rem3} below:
\begin{enumerate}
\item If $\check r^* = \check r ^{-1},$ then $\check r \check r^* = \check r^* \check r= \mbox{id}$ and $\check r \check r^*(x, y) = ( \sigma^z_{\hat \sigma_x(y)}(\hat \tau^z_{y}(x)),   \tau^{ z}_{\hat \tau^z_{y}(x)}(\hat \sigma^z_{x}(y)) ).$ Thus  $\sigma^{z}_{\hat \sigma_x^{z}(y)}(\hat \tau_y^{z}(x)) =x,$ $\tau^{z}_{\hat \tau^{ z}_{y}(x)}(\hat \sigma^{ z}_{x}(y)) = y.$ And vice versa if $\sigma^{z}_{\hat \sigma_x^{z}(y)}(\hat \tau_y^{z}(x)) =x, $\ $ \tau^{z}_{\hat \tau^{ z}_{y}(x)}(\hat \sigma^{ z}_{x}(y)) = y,$ then it automatically follows that $\check r^{*} = \check r^{-1}.$ Similarly,  $\check r^* \check r(x,y)= (x, y)$ leads to 
$\hat \sigma^{ z}_{\sigma_x^z(y)}(\tau_y^z(x)) =x, $ $\hat \tau^{ z}_{\tau^z_{y}(x)}(\sigma^z_{x}(y)) = y,$ and vice versa.

\item For the second part of the Proposition it suffices to show (\ref{k2}). Indeed,
\begin{eqnarray}
 \hat \sigma^z_{\sigma^z_x(y)}(\tau^z_y(x)) &=& -\sigma^z_x(y)\circ z^{-1}+ \sigma^z_x(y)\circ \tau_y^z(x)\circ z^{-1} \nonumber\\
&=&  
-\sigma^z_x(y)\circ z^{-1} + x \circ y \circ z^{-1} \nonumber\\
&=&  - (x\circ y - x\circ z +z)\circ z^{-1} +  x \circ y \circ z^{-1} \nonumber\\
&=&  -z\circ z^{-1} +x \circ z \circ z^{-1} -  x \circ y \circ z^{-1}+ x \circ y \circ z^{-1} = x. \nonumber
\end{eqnarray}
Also, $\hat \tau^{ z}_{\tau^z_{y}(x)}(\sigma^z_{x}(y))  = x^{-1} \circ \sigma^z_x(y) \circ \tau_y^z(x)= y. $
Similarly, we show
\begin{eqnarray}
\sigma^{z}_{\hat \sigma_x^{z}(y)}(\hat \tau_y^{z}(x))  &=& 
 \hat \sigma^z_x(y)\circ \hat \tau_y^z(x) -  \hat \sigma^z_x(y)\circ z +z \nonumber\\
&=& x \circ y - (-x \circ z^{-1} +x\circ y\circ z^{-1})\circ z
 +z\nonumber\\ &=&  x \circ y -x \circ y +x -z +z = x.\nonumber
\end{eqnarray}
And as above we immediately conclude that $\tau^{z}_{\hat \tau^{ z}_{y}(x)}(\hat \sigma^{ z}_{x}(y)) = y.$
\hfill \qedhere
\end{enumerate}
\end{proof}
\begin{remark}
Let $\check r(x,y)= (\sigma_x(y), \tau_y(x))$ and $\check r_{GV}(x,y)= (\hat \sigma_x(y), \hat \tau_y(x))$, such that:
\begin{equation}
\sigma_x(y) = x\circ y - x, \  \tau_y(x) = \sigma_x(y)^{-1}\circ x \circ y \quad \mbox{and} 
\quad \hat \sigma_x(y) = -x +x\circ y, \ \hat \tau_y(x) = \hat \sigma_x(y)^{-1}\circ x \circ y \nonumber
\end{equation} 
i.e, this is the special case $z=1$ ($\check r_{GV}$ is the solution of \cite{GV}).
Then via Proposition \ref{rem3} we immediately obtain that $ \check r_{GV} =  \check r^{-1}.$ Note that the solution $\check{r}$ is associated to {\it opposite skew braces} (see Theorem 4.1. in \cite{KochTruman}).
\end{remark}

\begin{example}[See \cite{BrzRybMer} Example 5.6 or \cite{BrzRyb:con} Example 3.15]\label{ex:fractions} 
Let us consider a set $\mathrm{Odd}:=\big\{\frac{2n+1}{2k+1}\ |\ n,k\in\mathbb{Z}\big\}$ together with two binary operations 
$(a,b)\overset{+_1}{\longmapsto}a-1+b$ and $(a,b)\overset{\circ}{\longmapsto}a\cdot b$, where $+,\cdot$ 
are addition and multiplication 
of rational numbers, respectively. The triple $(\mathrm{Odd},+_1,\circ)$ is a brace. By Lemma \ref{lem:inv}, 
the solution $\check r_z$ is involutive if 
and only if $a\cdot b-a\cdot z+z=a\cdot b-a+1$ if and only if $(z-1)\cdot (1-a)=0$, for all $ a,b\in B$. 
Therefore, for all $ z\not=1$, $\check r_{z}$ 
is non-involutive.  Moreover, $\check r_{z}=\check r_{w}$ if and only if $-a\cdot z+z=-a\cdot w+w$, that is if $z=w$.
\end{example}

\begin{example}[See \cite{BrzRyb:con} Corollary 3.14]\label{ex:cyclicbraces}
Let $\mathrm{U}(\mathbb{Z}/2^n\mathbb{Z})$ denote a set of invertible integers modulo $2^n$, for some $n\in \mathbb{N}$. 
Then a triple $(\mathrm{U}(\mathbb{Z}/2^n\mathbb{Z}),+_1,\circ )$ is a brace, where $a+_1b=a-1+b$ 
for all $ a,b\in \mathrm{U}(\mathbb{Z}/2^n\mathbb{Z})$, $+$ and $\circ $ are addition and multiplication of integer numbers modulo 
$2^n$, respectively.  Observe that $\check r_z=\check r_w$ if and only if $(a-1)\circ(w-z)=0 \pmod{2^n},$ for all $ a\in \mathrm{U}(\mathbb{Z}/2^n\mathbb{Z})$. 
\end{example}

\begin{example}[See \cite{BrzRybMer} Example 5.7]
Let us consider a ring $\mathbb{Z}/8\mathbb{Z}$. A triple $$\left(\mathrm{OM}:=\left \{\begin{pmatrix}
a & b\\
c & d\end{pmatrix}\ |\ a,d\in \{1,3,5,7\},\  b,c\in\{0,2,4,6\}
\right \},+_\mathbb{I},\circ\right)$$ is a brace, where $(A,B)\overset{+_{\mathbb{I}}}{\longmapsto} A-\mathbb{I}+B$, $(A,B)\overset{\circ}{\longmapsto} A\cdot B$, and $+,\cdot$ are addition and multiplication of two by two matrices over $\mathbb{Z}/8\mathbb{Z}$, respectively. Moreover one can easily check that two solutions $\check r_A$ and $\check r_B$ are equal if and only if 
$(D-\mathbb{I})\cdot(B-A)=0\pmod 8,$ for all $ D\in \mathrm{OM}$.
\end{example}

\begin{example}
Let $B$ be a two-sided brace and $S$ be a left skew brace. Then the product $B\times S$ is a left skew brace with operations 
given on coordinates. In this case,  for all elements of the form $(b,1)\in B\times S$, we can associate a solution, as in Theorem
\ref{prop:main} by defining $\check r_{(b,1)}((a,c),(d,e))=
((\sigma^b_a(d),\sigma^1_c(e)),(\tau^b_d(a),\tau^{1}_{e}(c)))$, for all $ (a,c),(d,e)\in B\times S$.
\end{example}

An interesting observation is that in contrast to the solutions given by skew braces and $z=1$,  in our parameter dependent solutions, $\tau^z_b$ or $\sigma^z_a$ are not necessarily right group actions, which follows from Lemma \ref{lem:action} below. Non-involutive parametric  solutions are obtained even in the case that the underlying algebraic structure is a brace. Furthermore, there does not necessarily have to exist a homomorphism of Yang-Baxter solutions (see Definition \ref{def11}) between $\check{r}_z$ and $\check{r}_{GV},$ where $\check{r}_{GV}$ is the canonical solution given by a left skew brace, as shown in Proposition \ref{pp2}. All this make our solutions of particular interest, and they certainly merit further investigation.

\begin{lemma}\label{lem:action}
Let $B$ be a left skew brace and $\check r_z(x,y)= (\sigma^z_{x}(y),\tau^z_{y}(x))$ be a solution defined in the Theorem \ref{prop:main}. Then $\tau^z:(B,\circ)\to \mathrm{Aut}(B)$, $b\mapsto \tau^z_b$ is a group action if and only if $a\circ z=z+a$. Similarly,  let $\check r^* = (\hat \sigma^z_{x}(y),\hat \tau^z_{y}(x))$ be the inverse solution of Proposition \ref{rem3},  then  $\hat \sigma^z:(B,\circ)\to \mathrm{Aut}(B)$, $a \mapsto \hat \sigma^z_a$  is a group action if and only if $a\circ z^{-1}=z^{-1}+a$. 
\end{lemma}
\begin{proof}
If $x\circ z=z+x$ for all $x\in B$, then $\tau^z=\tau^1$ and one can easily check that $\tau^1$ is an action. In the opposite direction,  let $\tau^z_a\tau^z_b=\tau^z_{b\circ a}. $ We first compute
$$
\begin{aligned}
&\tau^z_c(\tau^z_b(a))= \sigma^z_{\tau^z_b(a)}(c)^{-1}\circ \tau^z_{b}(a) \circ c =  \sigma^z_{\tau^z_b(a)}(c)^{-1}\circ \sigma^z_a(b)^{-1} \circ a \circ b \circ c,
\end{aligned}
$$
but $\sigma^z_{a}(b)\circ \sigma^z_{\tau^z_b(a)}(c) = \sigma_a^{z\circ z}(b\circ c)$ 
and hence $\tau^z_c(\tau^z_b(a)) = \tau_{b\circ c}^{z\circ z}(a).$ But due to our assumption that $\tau^z_b$ is a group action the following should hold $\forall a, b \in B$
\begin{equation} 
\tau^{z\circ z}_{b}(a) = \tau^z_{b}(a) \  \Rightarrow \ \sigma_{a}^{z\circ z}(b) = \sigma_a^{z}(b) \nonumber
\end{equation}
which leads to  $-a\circ z +z = -a \circ z \circ z + z \circ z$. For $a=z^{-1}$ we get
$z\circ z = z+z \Leftrightarrow z= -z^{-1}$. Then
\begin{eqnarray}
&& -a\circ z +z = -a \circ z \circ z + z \circ z\ \Rightarrow \nonumber\\
&& -a\circ z  +z = -a\circ (z+z) + z +z\  \Rightarrow\  \nonumber\\ && 
a -a\circ z +z =0\ \Rightarrow \  a\circ z = z+a. \nonumber
\end{eqnarray}
Similar proof holds for the second part of the Lemma.
\end{proof}

\begin{proposition} \label{pp2}
Let $(B, \circ, +)$ and $(S, \bullet,  +_s)$ be left skew braces,  and let $z \in B$ be $u$-distributive. Let also
$\check r : S\times S \to S\times S$ be the solution $\check r_{GV}(a,b) = ({-_s}  a  +_s a \bullet b,  (-_s a  +_s a\bullet  b)^{-1.} \bullet a \bullet b)$ ($-_s  a$ is the inverse with respect to $+_s,$ $a^{-1.}$ the inverse with respect to $\bullet$) and  $\check r_z : B \times B \to B \times B$ be the solution  $\check r_z(a,b) = ( - a \circ z^{-1}+ a\circ b\circ z^{-1},  (- a \circ z + a\circ b\circ z)^{-1} \circ a \circ b).$ If there exists a map $f: S \to B$ such that 
\begin{equation}
(f\times f) \check r_{GV} = \check r_z(f\times f), \label{themap}
\end{equation}
$1, z \in \mathrm{Im}(f)$  and $f(1) +z =z +f(1),$ then $z =- z^{-1}.$
\end{proposition}
\begin{proof}
From (\ref{themap}) it follows
\begin{eqnarray}
f(-_s a +_s  a\bullet b) = -f(a) \circ z^{-1} + f(a)\circ f(b) \circ z^{-1} \label{eq1}
\end{eqnarray}
By setting $(i)\ a=1,$ and $(ii)\ b=1$ (\ref{eq1}) we obtain equalities:
\begin{equation}
(i)~~f(b) = -f(1)\circ z^{-1} + f(1)\circ f(b) \circ z^{-1}  \quad \mbox{and} \quad   (ii)~~f(1) = -f(a)\circ z^{-1} + f(a)\circ f(1) \circ z^{-1} . \nonumber\\
\end{equation}
\begin{enumerate}
\item If $f(a) =z$ in equality (i), then  $f(1) \circ z = z \circ f(1).$\\
\item If $f(b) =z $ in equality (ii), then $f(1) \circ z^{-1} = f(1)-z.$\\
\item If $f(a) =1$ in equality (i), then $f(1) \circ z^{-1} = z^{-1} +f(1).$
\end{enumerate}

From the last two expressions above and the fact that $f(1) +z =z +f(1),$ we conclude that $z=-z^{-1}.$ That is for generic values of $z$ there exists no homorphism between the two distinct solutions.
\end{proof}

\begin{proposition}\label{ppm}
Let $(B,+,\circ )$, $(S,+_s,\bullet)$ be left skew braces, and let  $z \in B$ be $u$-distributive. Then we can consider the following two solutions
$\check r_1 : S\times S \to S\times S$ $ r_z : B \times B \to B \times B$ from the Proposition \ref{rem3}. Let $\eta\in B$ be such that $z+\eta\not=\eta\circ z$ and $f: S \to B$ be a function such that $\eta,1\in \mathrm{Im}(f)$, and $f(1)+\eta=\eta+f(1)$, then $f$ is not a homomorphism of the Yang-Baxter solutions, i.e.
\begin{equation}
(f\times f) \check r_1 \not=r_z(f\times f). 
\end{equation}
\end{proposition}
\begin{proof}
Let us assume ad absurdum that $f$ is a homomorphism of the Yang-Baxter solutions. Then
$$
\begin{aligned}
(f\times f)\check r_1&= r_z(f\times f),\\ 
f(-_sa+_sa\bullet b)&=f(a)\circ f(b)-f(a)\circ z+z,\\   f((\sigma_a(b))^{-1}\bullet a\bullet b)&=\sigma_{f(a)}^z(f(b))^{-1}\circ f(a)\circ f(b),
\end{aligned}
$$
for all $a,b\in S$. Observe that $f(1)$ commutes with all $f(a)$, for $a\in S$,
$$
\begin{aligned}
f(a)&=f(\tau_1(a))=\tau^z_{f(1)}(f(a))=\sigma^z_{f(a)}(f(1))^{-1}\circ f(a)\circ f(1)=f(\sigma_a(1))^{-1}\circ f(a)\circ f(1)\\ &=f(1)^{-1}\circ f(a)\circ f(1),
\end{aligned}
$$
and thus $f(1)\circ f(a)=f(a)\circ f(1)$. Moreover, by taking $b=1$, we get that 
$$
f(1)=f(a)\circ f(1)-f(a)\circ z+z\ \&\ -f(a)\circ f(1)+f(1)=-f(a)\circ z+z.
$$
Thus, $\sigma^z=\sigma^{f(1)}$ and $\tau^z=\tau^{f(1)}$. Then for all $a\in S$,
$$
f(a)=f(\sigma_1(a))=\sigma^{f(1)}_{f(1)}(f(a))=f(1)\circ f(a)-f(1)^2+f(1),
$$
and $-f(1)\circ f(a)+f(a)=-f(1)^2+f(1)$, which for $f(a)=1$ gives $f(1)^2=f(1)+f(1)$. By simple substitution we get $-f(1)\circ f(a)+f(a)=-f(1)$, and thus $f(a)+f(1)=f(1)\circ f(a)=f(a)\circ f(1)$.
Thus, finally,
$$
f(1)=f(a)\circ f(1)-f(a)\circ z+z\implies f(1)=f(a)+f(1)-f(a)\circ z+z,
$$
but for $f(a)=\eta$, we have that $f(1)+\eta=\eta+f(1)$, and 
$$
\eta\circ z=z+\eta,
$$
which contradicts with the assumption. Thus $f$ is not a homomorphism of Yang-Baxter solutions.
\end{proof}

\begin{corollary}
Observe that if $B$ is a left brace and $z+a\not=a\circ z$ for all $a\in B$, then there does not exists a skew brace $S$ such that $r_1:S\times S\to S\times S$ is isomorphic to $r_z:B\times B\to B\times B$, that is every surjection on the left brace satisfies all the assumptions of the map from Proposition \ref{ppm}.
\end{corollary}

We conclude the section by presenting the following example of a solution in which $a\mapsto \tau^z_a$ is not a right action of $(B,\circ)$, and the solution is not isomorphic to any other solution with parameter $1$.

\begin{example}\label{ex:par}
Let us consider a two-sided brace $U(\mathbb{Z}/16\mathbb{Z})$ as in Example \ref{ex:cyclicbraces}. Observe that in this case $\check r_7$ is not equivalent to $\check r_1$ as $5-1+7=11\pmod{16}$ and $5\circ 7=3\pmod{16}$. 
One can easily compute that 
$$
\tau^7_{15}(5)=5\pmod{16}\quad \&\quad \tau^7_{3}\tau^7_5(5)=13\pmod{16}.
$$
\end{example}

\section{Preliminaries on Quasi-bialgebras \&  Drinfeld twists }

\subsection{Quasi-bialgebras}
\noindent In this subsection we recall fundamental definitions on quasi-bialgebras and the notion 
of quasi-triangularity as well as some recent 
relevant results presented in \cite{DOGHVL}.
\begin{definition}{\label{definition1}} A quasi-bialgebra $\big ({\mathcal  A}, \Delta, \epsilon, \Phi, c_r, c_l \big )$ is a 
unital associative algebra ${\mathcal A}$ over some field $k$ with the following algebra homomorphisms:
\begin{itemize}
\item the co-product $\Delta: {\mathcal A} \to {\mathcal A} \otimes {\mathcal A}$
\item the co-unit $\epsilon: {\mathcal A} \to k$
\end{itemize}
together with the invertible element $\Phi\in {\mathcal A} \otimes {\mathcal A} \otimes {\mathcal A}$ 
(the associator) and the invertible elements $c_l, c_r \in {\mathcal A}$ (unit constraints),  such that:
\begin{enumerate}
\item $(\mbox{id} \otimes \Delta) \Delta(a)  = \Phi \Big ((\Delta \otimes \mbox{id}) \Delta(a)\Big ) \Phi^{-1},$ $\forall a \in {\mathcal A}.$
\item $\Big ((\mbox{id} \otimes \mbox{id} \otimes \Delta)\Phi \Big ) \Big ((\Delta \otimes \mbox{id} \otimes \mbox{id})\Phi \Big )= 
\Big (1\otimes \Phi \Big )  \Big ((\mbox{id} \otimes \Delta \otimes \mbox{id} )\Phi \Big ) \Big  (\Phi \otimes 1 \Big ).$
\item $(\epsilon \otimes \mbox{id})\Delta(a) = c_l^{-1} a c_l$ and $(\mbox{id} \otimes \epsilon)\Delta(a) = c_r^{-1} a c_r,$ 
$\forall a  \in {\mathcal A}.$
\item $(\mbox{id} \otimes \epsilon \otimes \mbox{id})\Phi = c_r \otimes c_l^{-1}.$
\end{enumerate}
\end{definition}
\noindent In the special case where $\Phi = 1 \otimes 1 \otimes 1$ one recovers a bialgebra, i.e.  co-associativity is restored.

Using the axioms of Definition \ref{definition1} further counit formulas 
for the associator and unit constraints can be derived \cite{DOGHVL}.
\begin{lemma} ({\it Lemma \cite{DOGHVL}}) {\label{lemma0}}
Let  $\big ({\mathcal  A}, \Delta, \epsilon, \Phi, c_r, c_l \big )$ be a quasi-bialgebra, then:
\begin{eqnarray}
(\epsilon \otimes \mbox{id} \otimes \mbox{id})\Phi = \Delta(c_l^{-1})(c_l\otimes 1), \quad
(\mbox{id} \otimes \mbox{id} \otimes \epsilon)\Phi =(1 \otimes c_r^{-1})\Delta(c_r)\quad  
\& \quad  \epsilon(c_l) = \epsilon(c_r). \nonumber 
\end{eqnarray}
\end{lemma} 

We introduce here some useful notation. 
Let $\pi: {\mathcal A} \otimes {\mathcal A} \to {\mathcal A} \otimes {\mathcal A}$ be the 
``flip'' map, such that $a\otimes b \mapsto b \otimes a$ $\forall a,b \in {\mathcal A},$ 
then we set $\Delta^{(op)} := \pi \circ \Delta.$ A quasi-bialgebra is called {\it cocommutative} if $\Delta^{(op)}= \Delta.$ 
We also  consider the general element $A = \sum_j a_j \otimes b_j \in {\mathcal A} \otimes {\mathcal A}$, 
then in the ``index'' notation we denote:
$A_{12} :=  \sum_j a_j\otimes b_j  \otimes 1,$ $A_{23} :=\sum_j 1\otimes   a_j\otimes b_j $ and $A_{13} :=  
\sum_j a_j \otimes 1\otimes b_j .$

The notion of quasi-triangularity for bialgebras extends to quasi-bialgebras \cite{Drintw, Kassel}.
\begin{definition}{\label{definition2}}  A quasi-bialgebra $\big ({\mathcal  A}, \Delta, \epsilon, \Phi \big )$ is called 
quasi-triangular (or braided) if an invertible element ${\mathcal R} \in {\mathcal A} \otimes {\mathcal A}$ 
(universal $R$-matrix) exists, such that
\begin{enumerate}
\item $\Delta^{(op)}(a) {\mathcal R} = {\mathcal R} \Delta(a),$ $\forall a \in {\mathcal A}.$
\item $(\mbox{id} \otimes \Delta){\mathcal R} = \Phi^{-1}_{231} {\mathcal R}_{13} \Phi_{213} {\mathcal R}_{12} \Phi^{-1}_{123}.$
\item $(\Delta \otimes \mbox{id}){\mathcal R} = \Phi_{312} {\mathcal R}_{13} \Phi_{132}^{-1} {\mathcal R}_{23} \Phi_{123}.$
\end{enumerate}
\end{definition}

From the axioms (1)-(3) of Definition \ref{definition2} and condition (3) of Definition \ref{definition1} 
one deduces that $(\epsilon \otimes \mbox{id}) {\mathcal R}= c_r^{-1} c_l$ and $(\mbox{id} \otimes \epsilon) {\mathcal R}= c_l^{-1} c_r$.
Moreover, by means of conditions (1)-(3) of Definition \ref{definition2}, it follows that ${\mathcal R}$ satisfies a non-associative
version of the quantum Yang Baxter equation (QYBE)
\begin{equation}
{\mathcal R}_{12} \Phi_{312} {\mathcal R}_{13} \Phi_{132}^{-1} {\mathcal R}_{23} \Phi_{123} =  
\Phi_{321}{\mathcal R}_{23}\Phi^{-1}_{231} {\mathcal R}_{13} \Phi_{213} {\mathcal R}_{12}. \label{modYBE} \nonumber
\end{equation}
\noindent In the case of $\Phi = 1 \otimes 1 \otimes 1$ one recovers the familiar QYBE 
and a quasi-triangular bialgebra.

In the following proposition we consider two special cases of the general algebraic setting for quasi-triangular 
quasi-bialgebras as described above. This setup,  introduced in \cite{DOGHVL}, will be particularly 
relevant to the findings of the next section.
\begin{proposition}(\it Proposition \cite{DOGHVL}){\label{lemma1}} Let $\big ({\mathcal  A}, \Delta, \epsilon, \Phi, {\mathcal R} \big )$  
be a quasi-triangular quasi-bialgebra, then the following two statements hold:
\begin{enumerate}
\item Suppose $\Phi$  satisfies the condition (in the index notation)
$\Phi_{213}{\mathcal R}_{12} = {\mathcal R}_{12} \Phi_{123},$ then
\begin{eqnarray}
(\mbox{id} \otimes \Delta){\mathcal R} =  \Phi^{-1}_{231} {\mathcal R}_{13}{\mathcal R}_{12}, \quad   
(\Delta \otimes \mbox{id}){\mathcal R} = {\mathcal R}_{13}  {\mathcal R}_{23} \Phi_{123}, \label{2} \nonumber
\end{eqnarray}
and the universal  ${\mathcal R}$ matrix satisfies the usual YBE. Also,
$(\epsilon \otimes \mbox{id}){\mathcal R} =  c_l,$ $~(\mbox{id} \otimes \epsilon){\mathcal R} = c_l^{-1} $ and $c_r =1_{\mathcal A}.$

\item Suppose $\Phi$  satisfies  the condition $ \Phi_{132} {\mathcal R}_{23} ={\mathcal R}_{23}\Phi_{123},$
then
\begin{eqnarray}
 (\mbox{id} \otimes \Delta){\mathcal R} =  {\mathcal R}_{13}{\mathcal R}_{12}\Phi^{-1}_{123}, \quad
 (\Delta \otimes \mbox{id}){\mathcal R} = \Phi_{312} {\mathcal R}_{13}  {\mathcal R}_{23},  \label{22} \nonumber
\end{eqnarray}
and the universal  ${\mathcal R}$ matrix satisfies the usual YBE. Also,
$(\epsilon \otimes \mbox{id}){\mathcal R} = c_r^{-1},$ $~(\mbox{id} \otimes \epsilon){\mathcal R} = 
c_r$ and $c_l=1_{\mathcal A}.$ 
\end{enumerate}
\end{proposition}
\noindent The detailed proof of the proposition above is given in \cite{DOGHVL}.

\subsection{Drinfeld twists} 

We recall here basic facts about the twisting of quasi-bialgebras 
and the generalized results obtained in \cite{DOGHVL}.
Drinfeld showed in \cite{Drinfeld, Drintw}  that
the property of being quasi-triangular  (quasi-)bialgebra is
preserved by using a suitable twist ${\mathcal F} \in {\mathcal A}  \otimes{\mathcal A}.$ Usually
whenever the notion of Drinfeld twist is discussed in the literature
a trivial action of the co-unit on the twist  is almost always assumed, i.e. 
$(\epsilon \otimes \mbox{id}){\mathcal F} =(\mbox{id} \otimes \epsilon){\mathcal F} = 1.$
In \cite{DOGHVL}
this condition is relaxed and the most general scenario is examined.
We should note that in \cite{Drintw}  certain types of
simple twists without this restricted counit action are used  to twist quasi-bialgebras 
with nontrivial unit constraints to quasi-bialgebras with trivial unit constraints
(see also detailed discussion on this point in \cite{DOGHVL}, where a more general framework is discussed).

The following \cite{DOGHVL} is a generalization of Drinfeld's result \cite{Drintw} 
on twists for quasi-bialgebras to the case of nontrivial unit constraints. 
Recall, without loss of generality,  that the unit constraints are expressed in terms 
of the associator Lemma \ref{lemma0}.
\begin{proposition}({\it Proposition \cite{DOGHVL, Drintw}}){\label{twist}} 
Let $\big ({\mathcal A},\Delta, \epsilon, \Phi, {\mathcal R} \big )$ be a 
quasi-triangular quasi-bialgebra and let ${\mathcal F} \in {\mathcal A} \otimes {\mathcal A}$ 
be an invertible element, such that
\begin{eqnarray}
&& \Delta_{\mathcal F}(a) = {\mathcal F} \Delta(a) {\mathcal F}^{-1}, ~~\forall a \in {\mathcal A} \\
&& \Phi_{\mathcal F} ({\mathcal F} \otimes 1) \big  ((\Delta \otimes \mbox{id}){\mathcal F} \big ) =  
 (1\otimes {\mathcal F})\big ((\mbox{id} \otimes \Delta){\mathcal F}\big ) \Phi \label{fi}\\
&& {\mathcal R}_{\mathcal F} = {\mathcal F}^{(op)} {\mathcal R} {\mathcal F}^{-1} \label{R}
\end{eqnarray}
where ${\mathcal F}^{(op)} := \pi({\mathcal F}),$ (recall $\pi$ is the flip map). 
Then  $\big ({\mathcal A},\Delta_{\mathcal F}, \epsilon, \Phi_{\mathcal F},  {\mathcal R}_{\mathcal F} \big )$ 
is also a quasi-triangular quasi-bialgebra.
\end{proposition} 
\noindent For the proof of Proposition \ref{twist} the general twist ${\mathcal F} = \sum_j f_j \otimes g_j$ is considered
(details on the proof of the proposition can be found in \cite{DOGHVL}).
In terms of the invertible elements $v:= \sum_{j}\epsilon(f_j) g_{j},$  $~w:= \sum_j \epsilon(g_j) f_j ,$ we have
$(\epsilon \otimes \mbox{id}){\mathcal F} = v,$ $\ (\mbox{id} \otimes \epsilon){\mathcal F} = w,$
so the trivial constraint is relaxed and indeed the most general scenario is regarded.

We introduce now a general frame, which will be appropriate when examining quantum 
algebras emerging from non-degenerate, set-theoretic solutions of the YBE presented in the next section,
compatible also with the analysis in \cite{Doikoutw, DOGHVL}. The following lemma is a generalization 
of Proposition 1.5 and Remark 1.9 of \cite{DOGHVL}, suitable
for the purposes of the next section.

\begin{lemma}{\label{remdr}} :
Let $\big ({\mathcal A},\Delta, \epsilon,\Phi, {\mathcal R} \big )$
and $\big ({\mathcal A},\Delta_{\mathcal F}, \epsilon, \Phi_{\mathcal F}, {\mathcal R}_{\mathcal F} \big )$ 
be quasi-triangular quasi-bialgebras  
and let the conditions of 
Proposition \ref{lemma1} hold. We recall also the useful notation: 
${\mathcal F}_{1,23} :=(\mbox{id} \otimes \Delta){\mathcal F},$ 
${\mathcal F}_{12,3} :=(\Delta \otimes \mbox{id}){\mathcal F},$ and by the quasi-bialgebra axioms 
${\mathcal F}_{21,3} {\mathcal R}_{12} = {\mathcal R}_{12} {\mathcal F}_{12,3},$ $ {\mathcal F}_{1,32} 
{\mathcal R}_{23} = {\mathcal R}_{23} {\mathcal F}_{1,23}. $ According to  Proposition \ref{lemma1} we distinguish two cases:
\begin{enumerate}
\item  If the associators satisfy 
\begin{equation}
\Phi_{213} {\mathcal R}_{12} = {\mathcal R}_{12}\Phi_{123},   \qquad \Phi_{{\mathcal F}213} 
{\mathcal R}_{12} = {\mathcal R}_{12}\Phi_{{\mathcal F}123}, \label{con1}
\end{equation} 
and $\Phi_{\mathcal F} $ commutes with $F_{12},$ then (in the index notation)
the condition (\ref{fi}) can be re-expressed as 
$~{\mathcal F}_{123} := {\mathcal F} _{23}{\mathcal F}_{1,23} = {\mathcal F}_{12} {\mathcal F}_{12,3}^*, $
where ${\mathcal F}_{12.3}^* = \Phi_{\mathcal F} {\mathcal F}_{12,3}\Phi^{-1}.$
We also  deduce that
${\mathcal F}^*_{21,3} {\mathcal R}_{12} = {\mathcal R}_{12} {\mathcal F}^*_{12,3}.$ 
This  is compatible also with the first part of Proposition \ref{lemma1}.
\item If the associator satisfies 
\begin{equation}
\Phi_{132} {\mathcal R}_{23} = {\mathcal R}_{23}\Phi_{123}, \qquad  \Phi_{{\mathcal F}132} 
{\mathcal R}_{23} = {\mathcal R}_{23}\Phi_{{\mathcal F}123}, \label{con2}
\end{equation}
and $\Phi_{\mathcal F} $ commutes with $F_{23},$  then 
then condition (\ref{fi}) is re-expressed as 
$~{\mathcal F}_{123} := {\mathcal F} _{23}{\mathcal F}^*_{1,23} = {\mathcal F}_{12} {\mathcal F}_{12,3}, \label{d2}$
where ${\mathcal F}_{1,23}^* = \Phi^{-1}_{\mathcal F} {\mathcal F}_{1,23}\Phi.$
We also  deduce that
$ {\mathcal F}^*_{1,32} {\mathcal R}_{23} = {\mathcal R}_{23} {\mathcal F}^*_{1,23}.$ 
This is compatible with the second part of Proposition \ref{lemma1}.
\end{enumerate} 
\end{lemma}
\begin{proof}
The proof is straightforward due to Proposition \ref{twist} and constraints (\ref{con1}),  (\ref{con2}).
\end{proof}

\section{Non-involutive solutions of the YBE \& Drinfeld twists}

\noindent We move on to the study of non-involutive set-theoretic solutions of the YBE, 
admissible Drinfeld twists and the associated quasi-bialgebras.  From now on we work over the field $k = {\mathbb C}.$
We first review some fundamental results on the admissible Drinfeld twists for involutive
set-theoretic solutions of the YBE derived in \cite{Doikoutw} and we use these 
twists to extend these results to the non-involutive case.

Let $X$ be a set with $n$ elements and $\check{r}_z:X \times X\to X \times X$ 
be a solution of the set-theoretic braid equation, given in the previous section.  It is convenient 
for the purposes of this section to 
consider a free 
vector space $V= \mathbb{C}X$ of dimension equal to the cardinality of $X$. 
Let  ${\mathbb B} = \{e_x\}_{x\in X}$ be the basis of $V$ and ${\mathbb B}^* = \{e^*_x\}_{x\in X}$ 
be the dual basis: $e_x^* e_y= \delta_{x,y }.$
Let also $f\in V \otimes V$ be expressed as $f= \sum_{x,y\in X} f(x,y) e_x 
\otimes e_y,$ then  the set-theoretic solution is a map $r_z: V \otimes V\to V\otimes V,$ such that
$\ (\check r_z f)(x,y) = f(\sigma^z_x(y), \tau^z_y(x)),$
$\forall x,y\in X$; equivalently,
$$
\check r_z=\sum_{x,y\in X}e_{x,\sz{x}{y}}\otimes e_{y,\tz{y}{x}},
$$
where $e_{x,y}$ are  defined as $e_{x,y} =e_x  e_y^*.$  
In this construction $e_x$ corresponds to a $n$-column vector with 1 at the $x$ position and 0 
elsewhere,  $e_x^*= e^T$ ($^T$ denotes transposition) 
and $e_{x,y}$ is an $n\times n$ matrix with one identity 
entry in $x$-row and $y$-column, and zeros elsewhere, i.e $(e_{x,y})_{z,w}=\delta_{x,z}\delta_{y,w},$ 
and  $\check r_z$ is a $n^2\times n^2$ matrix. \footnote{This can be formally extended to the 
countably infinite case $(n \to \infty),$ 
provided that finite norm elements of the space are considered. 
The infinite case and possible connections with orthogonal polynomials will be discussed in detail elsewhere.}

The matrix, $\check r_z$ satisfies the braid equation:
$$
(\id\otimes \check r_z)(\check r_z\otimes \id)(\id\otimes \check r_z)=
(\check r_z\otimes \id)(\id\otimes \check r_z)(\check r_z\otimes \id).
$$

We also recall $r_z= {\mathcal P} \check r_z$, where ${\mathcal P} = \sum_{x, y\in X} e_{x,y}\otimes e_{y,x}$ 
is the permutation operator, then $r_z$ has the explicit form
\begin{equation}
r_z=\sum_{x, y\in X} e_{y, \sigma^z_x(y)}\otimes e_{x, \tau^z_y(x)}, \label{brace2} \nonumber
\end{equation}
i.e.  $(r_zf)(y,x) = f(\sigma^z_x(y), \tau^z_y(x))$, and satisfies the YBE.

\noindent Henceforth, we will drop the $z$-index in $\check r_z, \ \sigma^z, \ \tau^z $
for brevity,  although it is always implied.

Before we start discussing the admissible twists we recall the definitions of two quadratic 
algebras ${\mathcal A}$ and ${\mathcal Q}$ associated to set-theoretic solutions, which arise 
from the FRT (Faddeev, Reshetikhin and Takhtajan) construction \cite{FadTakRes}. 
This will be useful for our considerations later in the text.
Indeed, from the FRT construction
we recall:

\begin{definition}
Given a solution of the braid equation $\check r: V \to V$ ($V = {\mathbb C}X$),  
for a finite set $X,$ the associated quantum algebra ${\mathcal A}$ is a quotient of a 
free associative $C$-algebra, generated by $\{L_{z,w}|\ x,w \in X\},$ and relations
\begin{equation}
\check r_{12}\ L_1\ L_2 = L_1\ L_2\ \check r_{12}, \label{RTT} \
\end{equation} 
where  $\ L  = \sum_{x,y \in X} e_{x,y} \otimes L_{x,y}\in 
\mbox{End}(V) \otimes {\mathcal A}$. 
Recall  the {\it index notation}: $\check r_{12} = \check r \otimes 1_{\mathcal A}$ and
$L_1 = \sum_{z, w \in X} e_{z,w} \otimes I \otimes L_{z,w}, $ $\  L_2= \sum_{z, w \in X} I  \otimes  e_{z,w}  \otimes L_{z,w}.$
\end{definition}

From the fundamental relation (\ref{RTT}) for the set-theoretic solution of the braid equation \cite{ESS}:
 \begin{eqnarray}
L_{x, \hat x} L_{y, \hat y} = L_{\sigma_x(y), \sigma_{\hat x}(\hat y)} L_{\tau_y(x), \tau_{\hat y}(\hat x)}.
 \label{q11}
\end{eqnarray} 
See also \cite{DoiSmo1} for the algebra associated to the Baxterized $r$-matrix.

\begin{definition} 
Given a solution of the YBE $r: V \to V,$ the quadratic algebra ${\mathcal Q}$ is generated by $\{q_x|\ x \in X\}$  and relations
\begin{equation}
 r_{12}\ q_1\ q_2 = q_2\ q_1, \label{RTT2} 
\end{equation} 
where $\ q  = e_x \otimes q_x \in V \otimes {\mathcal Q}.$ Also, $r_{12} =r \otimes 1_{\mathcal A},$
$\ q_1 = \sum_{z, w \in X} e_{x} \otimes I \otimes q_{x},$ $\ q_2= \sum_{x \in X} I  \otimes  e_{x}  \otimes q_{x}.$
\end{definition}

The quadratic relation (\ref{RTT2}) for the set-theoretic solution of the YBE implies 
 \begin{eqnarray}
q_{x} q_{y} = q_{\sigma_x(y)} q_{\tau_y(x)}, \label{qalg}
\end{eqnarray} 
also obtained in \cite{ESS}.

It is worth noting that if we conveniently re-express $L= \sum_{x,y\in X} e_{y, \sigma_{x}(y)} 
\otimes L_{x, \tau_{y}(x)}$, imitating the form of the  set-theoretic $r$-matrix,  we conclude via (\ref{RTT}) that the relations of the corresponding quantum algebra ${\mathcal A}$ are,
\begin{equation}
 L_{\eta, \tau_{x}(\eta)} L_{\hat x, \tau_y(\hat x)} =  L_{\hat \eta, \tau_{\sigma_x(y)}(\hat \eta)}
L_{\bar x, \tau_{\tau_y(x)}(\bar x)}, \label{new}
\end{equation}
subject to: $\sigma_{\hat \eta}(\sigma_x(y)) = \sigma_{\sigma_{\eta}(x)}(\sigma_{\hat x}(y))$ and 
$ \tau_{\sigma_{\hat x}(y)}(\sigma_{\eta}(x)) =  \sigma_{\bar x}(\tau_y(x)).$
The quantum algebras ${\mathcal A}$ with relations (\ref{q11}), (\ref{new}) are also closely related to the reflection algebra for set-theoretic solutions \cite{DoiSmo2} and will be discussed in detail in future works.

We now briefly recall the notion of admissible Drinfeld twists for set-theoretic solutions.
It was shown in \cite{Doikoutw} that all involutive set-theoretic solutions 
can be obtained from the permutation operator via suitable twists.
Two distinct admissible twists $F,  \hat F \in \mbox{End} (V\otimes V)$ were identified in \cite{Doikoutw}. 
Indeed,  consider the invertible elements ${\mathbb V}_x 
=\sum_{y \in X} e_{\sigma_{x}(y), y}$ and ${\mathbb W}_y =\sum_{x\in X}e_{\tau_y(x), x},$ then 
\begin{equation}
F = \sum_{x\in X} e_{x,x} \otimes {\mathbb V}_x, \qquad 
\hat F =  \sum_{y\in X} {\mathbb W}_y\otimes e_{y,y}.  \label{twistsa}
\end{equation}
These are both admissible twists in the involutive case as was shown in \cite{Doikoutw}.

We will show in what follows that both twists (\ref{twistsa})
are still admissible in the case of non-involutive solutions and we will identify
the twisted $r$-matrices.  Note that, similarly to the involutive case analysed in \cite{DOGHVL}, 
the underlying algebra is not a quasi-triangular bialgebra, as one would expect, 
but a quasi-triangular quasi-bialgebra. 
We shall now prove the following useful lemma.
\begin{lemma}{\label{basicla}} 
Let  $\check r: V \otimes V\to 
V \otimes V$ be a non-degenerate, non-involutive, set-theoretic solution of the braid equation,  
$\check r = \sum_{x,y \in X} e_{x, \sigma_{x}(y)} \otimes e_{y, \tau_y(x)}.$ Suppose that for every $x\in X,$ the elements  
\begin{equation}
{\mathbb V}_{x} = \sum_{y \in X} e_{\sigma_{x}(y), y}, \quad
{\mathbb W}_x =\sum_{\eta\in X}e_{\tau_x(\eta), \eta}, 
\end{equation}
satisfy,
\begin{equation}
\Delta({\mathbb V}_{\eta}) = \sum_{x,y\in X}e_{\sigma_{\eta}(x), x} \otimes
e_{\sigma_{\tau_{x}(\eta)}(y), y}, \quad \Delta({\mathbb W}_{y}) = 
\sum_{\eta, x \in X}  e_{\tau_{\sigma_x(y)}(\eta), \eta} \otimes e_{\tau_{y(x), x}}.
\label{twistdelta}
\end{equation}
 Then$\ \Delta({\mathbb Y}_x) \check r = 
\check r\Delta({\mathbb Y}_x)  , \quad 
{\mathbb Y}_x \in \{ {\mathbb V}_x, {\mathbb W}_x\}. $
\end{lemma}
\begin{proof} 
It is computationally easier to show that $\big [\Delta({\mathbb Y}_x^{-1}),\ \check r\big]  =0$, 
where specifically $\Delta({\mathbb Y}_x^{-1}) = \Delta({\mathbb Y}_x)^T,$ 
(also ${\mathbb Y}_x^{-1}= {\mathbb Y}_x^T$)  and $^T$ denotes transposition.
We compute explicitly
\begin{eqnarray}
\check r \Delta({\mathbb V}_x^{-1}) &=&  \sum_{y,z \in X} e_{y,\sigma_x(\sigma_y(z))}\otimes  
e_{z, \sigma_{\tau_{\sigma_y(z)}(x)}(\tau_z(y))}\nonumber\\
\Delta({\mathbb V}_x^{-1})  \check r &=&  \sum_{y,z \in X} e_{y,\sigma_{\sigma_x(y)}
(\sigma_{\tau_y(x)}(z))}\otimes  e_{z, \tau_{\sigma_{\tau_y(x)}(z)}(\sigma_x(y))}. \nonumber
\end{eqnarray}
Due to conditions (\ref{C1}) and (\ref{C3}) for the set-theoretic solution $\check r$ we conclude that for all $x\in X,$
$$\big[ \check r,\ \Delta({\mathbb V}_x^{-1}) \big] = 0.$$ 

\noindent Similarly, by conditions (\ref{C2}) and (\ref{C3}) we show that for all $y \in X,$ $$\big[\check r, \Delta({\mathbb W}_y^{-1})  \big] 
= 0,$$ and this concludes our proof.
\end{proof}
It is worth noting that the form and the coproduct structure of the  elements ${\mathbb V}_x,\ {\mathbb W}_x$
are inspired by tensor representations of the RTT algebra (\ref{RTT}) (we refer the interested reader to 
\cite{FadTakRes, Doikoutw}). It is also worth recalling at this point the algebra ${\mathcal Q}$ 
generated by $q_x$ $x\in X$ (\ref{qalg}). It turns out that ${\mathbb V}_x$ and ${\mathbb W}_x^T$ 
($^T$ denotes transposition), defined earlier,  are $n$-dimensional representations of 
${\mathcal Q},$ i.e.  $q_x \mapsto {\mathbb U}_x, $ where ${\mathbb U}_x \in \{ {\mathbb V}_x,\ {\mathbb W}^T_x\},$ 
indeed:
\begin{lemma} The $n\times n$ matrices ${\mathbb V}_x =\sum_{y \in X} e_{\sigma_{x}(y), y}$ 
and ${\mathbb W}^T_x =\sum_{\eta\in X}e_{\eta, \tau_x(\eta)},$ for all $ x \in X,$  satisfy the algebraic relations (\ref{qalg}). 
\end{lemma}
\begin{proof} 
The proof is based on straightforward computation and use of the properties of skew braces:
\begin{eqnarray}
{\mathbb V}_x {\mathbb V}_y = \sum_{z\in X} e_{\sigma_{x}(\sigma_y(z)),z} = 
\sum_{z\in X} e_{\sigma_{x\circ y}(z), z} = {\mathbb V}_{\sigma_x(y)} {\mathbb V}_{\tau_y(x)}, \nonumber
\end{eqnarray}
and similarly, ${\mathbb W}^T_x {\mathbb W}^T_y = {\mathbb W}^T_{\sigma_x(y)} {\mathbb W}^T_{\tau_y(x)}.$
\end{proof}
\noindent Also, the coproducts $\Delta({\mathbb V}_x),\Delta({\mathbb W}^T_x)$ defined in (\ref{twistdelta}),  satisfy the algebraic relations (\ref{qalg}), i.e.  $\Delta$ is indeed an algebra homorphism.

We may now extend the results of \cite{Doikoutw, DOGHVL}  to the non-involutive case.
\begin{proposition} {\label{propg1}}(Proposition \cite{Doikoutw})
We recall the following quantities introduced in \cite{Doikoutw}:
\begin{eqnarray}
F_{1,23} = \sum_{x, y, \eta \in X} e_{\eta, \eta} \otimes e_{\sigma_{\eta}(x), x}
 \otimes e_{\sigma_{\tau_x(\eta)}(y), y} , \quad F^*_{12,3} =\sum_{x, y, \eta \in X} 
e_{\eta, \eta} \otimes e_{x,x}\ 
\otimes e_{\sigma_{\eta}(\sigma_{x}(y)), y} \nonumber 
\label{cop2a}
\end{eqnarray}
\begin{eqnarray}
\hat F^*_{1,23} = \sum_{x, y, \eta \in X} e_{\tau_y(\tau_x(\eta)), \eta} \otimes e_{x,x}
 \otimes e_{y,y}, \quad \hat F_{12,3} =\sum_{x, y, \eta \in X} e_{\tau_{\sigma_x(y)}(\eta), \eta} 
\otimes e_{\tau_y(x), x }\
\otimes e_{y, y}, \nonumber \label{cop2ab}
\end{eqnarray}
subject to \ref{C1})-(\ref{C3}).  Let also $\check r: V \otimes V\to
V \otimes V$ be the set-theoretic solution of the braid equation $\check r = \sum_{x, y \in X } 
e_{x, \sigma_{x}(y)} \otimes e_{y, \tau_{y}(x)}$, then 
\begin{eqnarray}
\check r_{12} F^*_{12,3} = F^*_{12,3} \check r_{12},
\quad  \check r_{23} F_{1,23} = F_{1,23}\check r_{23} \quad \& \quad \check r_{12} \hat F_{12,3} = 
\hat F_{12,3} \check r_{12},  
\quad  \check r_{23} \hat F^*_{1,23} = F^*_{1,23}\check r_{23}. \label{comut1} \nonumber
\end{eqnarray}
\end{proposition}
\begin{proof}
The proof goes along the same lines as in \cite{Doikoutw}.  Indeed, we first observe,  recalling 
Lemma \ref{basicla} and the definition of $F,$
 that $F_{1,23} = (\mbox{id} \otimes \Delta)F,$ whereas $F_{12,3}^* \neq (\Delta \otimes\mbox{id})F$. 
Hence, we immediately conclude that $\big  [F_{1,23}, \ \check r_{23}\big ]  =0;$ also by explicit 
computation and bearing in mind (\ref{C1}) we show that $\big  [F_{12,3}^*, \ \check r_{12}\big ]  =0.$ 

Similarly,  due to Lemma \ref{basicla} and the definition of $\hat F,$ 
$\hat F_{12,3} = (\Delta \otimes \mbox{id})\hat F,$ 
whereas $\hat F^*_{1,23} \neq (\mbox{id} \otimes \Delta)\hat F.$ Then it immediately follows 
that $\big  [\hat F_{12,3}, \ \check r_{12}\big ]  =0;$
also by explicit computation and  by recalling (\ref{C2}) we conclude that 
$\big  [\hat F^*_{1,23}, \ \check r_{23}\big ]  =0,$ (see also relevant Lemma \ref{remdr}).
\end{proof}

The admissibility of the twist for involutive solutions was  proven in Proposition 3.15 in 
\cite{Doikoutw}.  In the following proposition we generalize this result
in the non-involutive scenario.
\begin{proposition}{\label{cocycle}}(Proposition \cite{Doikoutw})
 Let ${\mathcal T}_{12} ={\mathcal T}  \otimes \mbox{id},$ 
${\mathcal T}_{23} =\mbox{id} \otimes {\mathcal T},$ where ${\mathcal T} \in \{F,\ \hat F\}$ and
$F, \ \hat F$ are given in (\ref{twistsa}).
Let also  $F^*_{12,3}, F_{1,23}$ and $\hat F_{12,3},  \hat F^{*}_{1,23}$ defined in Proposition \ref{propg1}. Then 
\begin{equation}
F_{123}:=F_{12} F^*_{12,3} =F_{23}F_{1,23}, \qquad \hat F_{123}:=\hat F_{12} \hat F_{12,3} 
=\hat F_{23}\hat F^*_{1,23}.\nonumber
\end{equation}
\end{proposition}
\begin{proof}
\noindent The proof again goes along the same lines as in \cite{Doikoutw}.
By substituting the expressions for $F_{12},\ F_{23}$,  $F^*_{12,3}$ and $F_{1,23}$ 
and recalling that conditions (\ref{C1})-(\ref{C3}) hold, we obtain,
\begin{equation}
F_{123}:=F_{12} F^*_{12,3} =F_{23}F_{1,23}= \sum_{\eta,x,y \in X}  
e_{\eta, \eta} \otimes e_{\sigma_{\eta}(x), x } \otimes   e_{\sigma_{\eta}(\sigma_{x}(y)),y}. \nonumber
\end{equation}
Similarly,  from expressions  $\hat F_{12},\ \hat F_{23}$, $\hat F_{12,3}$ and $\hat F^*_{1,23}$ 
and using  conditions (\ref{C1})-(\ref{C3}) we obtain,
\begin{equation}
\hat F_{123}:=\hat F_{12} \hat F_{12,3} =\hat F_{23}\hat F^*_{1,23}= \sum_{\eta,x,y \in X}  
e_{\tau_y(\tau_x(\eta)), \eta} \otimes e_{\tau_y(x), x } \otimes e_{y,y}. \nonumber
\hfill \qedhere
\end{equation}
\end{proof}
\begin{remark}{\label{rem4}}
We should note that no matter what the action of the counit on ${\mathbb V}_x,\ {\mathbb W}_x$ 
is, we deduce that $(\mbox{id} \otimes \epsilon) F = \sum_{\eta \in X} \epsilon({\mathbb V}_{\eta})e_{\eta, \eta}$ 
(it becomes the $n\times n$ identity matrix $1_n$ if 
$\epsilon({\mathbb V}_{\eta}) =1,$ for all $ \eta \in X$)  and 
$(\epsilon \otimes \mbox{id}) F  = \sum_{\eta \in X} \epsilon(e_{\eta, \eta}) {\mathbb V}_{\eta} \neq 1_n.$ 
Similarly, $(\epsilon \otimes \mbox{id}) \hat F  =\sum_{x \in X} \epsilon({\mathbb W}_{x})e_{x,x}$ 
(it becomes the $n\times n$ identity matrix if $\epsilon({\mathbb W}_x) =1,$ for all $ x \in X$) and 
$(\mbox{id} \otimes \epsilon) \hat F = \sum_{x\in X} \epsilon(e_{x,x}) {\mathbb V}_x\neq 1_n.$
\end{remark}

\begin{remark} Propositions \ref{propg1}, \ref{cocycle} are essential in showing that if $\check r$ 
is a solution of the braid equation, 
then $\check r_F$ also is (see e.g.  \cite{Drin, Doikoutw} and Proposition \ref{twist}).  Indeed,  if $\check r$ 
satisfies the braid equation, then 
by acting from the left of the braid equation with $F_{123}$ and from the right with $F_{123}^{-1}$ 
and recalling that $F\check r F^{-1} = \check r_F,$ as well as Propositions \ref{propg1}, \ref{cocycle}, 
we conclude that $\check r_F$ also satisfies the braid equation.
\end{remark}

We now identify the twisted $\check r$-matrices, which are also solutions of the braid equation, 
as well as the  associated twisted coproducts.
\begin{lemma}{\label{basiclb}} 
Let  $\check r: V \otimes V\to V \otimes V$ be the set-theoretic solution of the braid equation.
Let also ${\mathbb V}_{x} = \sum_{y \in X} e_{\sigma_{x}(y), y}$ and 
${\mathbb W}_x =\sum_{\eta\in X}e_{\tau_x(\eta), \eta}, $ for all $ x \in X,$  
with coproducts defined in (\ref{twistdelta}).
Then $\Delta_{\mathcal T}({\mathbb Y}_x) \check r_{\mathcal T} = 
\check r_{\mathcal T} \Delta_{\mathcal  T}({\mathbb Y}_x),$
where  ${\mathbb Y}_x \in \{ {\mathbb V}_x, \ {\mathbb W}_x\},$ ${\mathcal T} \in \{ F, \ \hat F \},$ 
\begin{equation}
\Delta_F({\mathbb V}_x) = {\mathbb V}_x \otimes {\mathbb V}_x,  \quad \Delta_F({\mathbb W}_{y}) 
= \sum_{\eta, x\in X} e_{\tau_{\sigma_x(y)}(\eta), \eta} \otimes 
e_{\tau_{\sigma_{\tau_x(\eta)}(y)}(\sigma_{\eta}(x)),\sigma_{\eta}(x)}\label{delta} \nonumber
\end{equation}
\begin{equation}
\Delta_{\hat F}({\mathbb V}_{\eta}) = \sum_{x, y\in X}
e_{\sigma_{\tau_{\sigma_x(y)}(\eta)}(\tau_y(x)), \tau_y(x)}\otimes e_{\sigma_{\tau_x(\eta)}(y), y}  \quad 
\Delta_{\hat F}({\mathbb W}_{y}) = {\mathbb W}_y \otimes {\mathbb W}_y \label{delta2} \nonumber
\end{equation}
and the twisted matrices read as
\begin{equation}
\check r_F = \sum_{x, y \in X} e_{x, \sigma_x(y)}\otimes e_{\sigma_x(y), \sigma_{\sigma_x(y)}(\tau_y(x))} \quad \&  \quad
\check r_{\hat F} = \sum_{x, y \in X}e_{\tau_y(x), \tau_{\tau_y(x)}(\sigma_{x}(y))} \otimes e_{y, \tau_y(x)}. \nonumber
\end{equation}
\end{lemma}
\begin{proof} The proof immediately follows after recalling expressions ${\mathbb V}_x,  {\mathbb W}_x,$ 
$ F,  \hat F, $ $\check r,$ $\Delta({\mathbb V}_{x}),$ $\Delta({\mathbb W}_{x}),$ conditions (\ref{C1})-(\ref{C3}) 
and by explicit computation. 
\end{proof}
\begin{remark} Notice that in the involutive case both twisted $\check r$-matrices reduce to 
the permutation operator as expected. 
Also, it follows from Lemma \ref{basiclb} that when we twist with $F$ the invertible elements 
${\mathbb V}_x \in \mbox{End}(V)$ form a family of group like elements,
i.e.  $\Delta_F({\mathbb V}_x) = {\mathbb V}_x \otimes {\mathbb V}_x$ and 
$\epsilon({\mathbb V}_x) = 1,$ for all $ x\in X.$ Similarly,  when twisting with $\hat F$ the invertible elements  
${\mathbb W}_x \in \mbox{End}(V)$ form a family of group like elements,  i.e.
 $\Delta_{\hat F}({\mathbb W}_{x}) = {\mathbb W}_x \otimes {\mathbb W}_x$ 
and $\epsilon({\mathbb W}_x) = 1,$ for all $ x\in X$ (see Remark \ref{rem4}).
\end{remark}

\begin{remark}
Recall that $r = {\mathcal P} \check r$ 
is the set-theoretic solution of the QYBE. The established equalities $\epsilon({\mathbb V}_x) =\epsilon({\mathbb W}_x) = 1,$ for all $x\in X,$ and the coproducts (\ref{twistdelta}) given in Lemma \ref{basicla}, imply the following:
\begin{enumerate}
\item the coproduts (\ref{twistdelta}) are not coassociative.
\item $(\epsilon \otimes \mbox{id}) \Delta({\mathbb V}_{\eta}) = \sum_{x \in X}\epsilon(e_{\sigma_{\eta}(x), x})
{\mathbb V}_{\tau_{x}(\eta)},$ $\ (\mbox{id} \otimes \epsilon)\Delta({\mathbb V}_{\eta}) = {\mathbb V}_{\eta}$ 
and\\ $(\epsilon \otimes \mbox{id}) \Delta({\mathbb W}_{y}) = {\mathbb W}_{y}, $ $\ (\mbox{id} 
\otimes \epsilon)\Delta({\mathbb W}_{y}) = \sum_{x \in X}\epsilon(e_{\tau_{y}(x), x}) {\mathbb W}_{\sigma_{x}(y)}.$
\item $(\Delta \otimes \mbox{id})r \neq r_{13} r_{23}$ and  $( \mbox{id} \otimes \Delta )r \neq r_{13} r_{12}.$ 
\item  The set-theoretic $r$-matrix satisfies the YBE. 
\end{enumerate}
All the above lead to the conclusion that,  as in the involutive case \cite{DOGHVL}, the underlying 
quantum algebra is a quasi-triangular 
quasi-bialgebra (see also Definitions \ref{definition1}, \ref{definition2} and Proposition \ref{lemma1}).
Similar comments can be made for the twisted $r$-matrices.
\end{remark}
Relevant results on the extension of the results of \cite{Doikoutw}  
about the admissibility of the twist $F$  (\ref{twistsa}) 
to the non-involutive case are presented in \cite{Ghobadi}.
And although an example of a twisted $r$-matrix is also presented neither the issue of the twisted 
coproducts of the families of operators ${\mathbb V}_x, {\mathbb W}_x,$ 
which play a crucial role in characterising  the  associated quantum algebra as a quasi-bialgebra, 
nor the issue of the quantum algebra being a quasi-bialgebra are discussed.
In fact,  it is implicitly regarded  that there is an underlying Hopf structure,  
but as we have shown  this is not the case.

The next important step in this frame is to identify  the associators $\Phi, \ \Phi_{\mathcal T}$ 
and demonstrate the full structure of the underlying quasi-bialgebras for all set-theoretic solutions. 
This will indeed confirm the main conjecture  of \cite{DOGHVL} that the underlying quantum algebra 
for involutive set-theoretic solutions is a quasi-triangular quasi-bialgebra. In the involutive 
case $\Phi_{\mathcal T} = 1\otimes 1 \otimes 1$, so the task in this situation is to identify
 $\Phi,$ but this will be discussed in future works. 
We should also note that the notion of the antipode and the quasi-Hopf algebra is briefly 
discussed for the involutive case in \cite{DOGHVL} for some special set-theoretic solutions,  
however further study in this direction is required.  The Baxterization of the non-involutive solutions and  the derivation of a universal $R$-matrix are among the most challenging questions in this context and will be addressed in future investigations.

\subsection*{Acknowledgments}
\noindent  The authors would like to thank \L ukasz Kubat for comments to the manuscript during ``The algebra of the Yang-Baxter equation" conference in Bedlewo 2022. Support from the EPSRC research grant EP/V008129/1 is acknowledged.

\end{document}